\documentclass [leqno,10pt]{amsart}
\usepackage{amsfonts}
\usepackage{amsmath}
\usepackage{amssymb}
\usepackage{graphicx}
\newtheorem{theorem}{Theorem}[section]
\theoremstyle{plain}

\newtheorem{corollary}{Corollary}[section]

\newtheorem{definition}{Definition}[section]
\newtheorem{example}{Example}[section]

\newtheorem{lemma}{Lemma}[section]

\numberwithin{equation}{section} \textheight  22 true cm \textwidth  15 true cm  
\begin{document}
\title[Conformal Semi-invariant submersions]{%
CONFORMAL SEMI-INVARIANT SUBMERSIONS}
\author[Akyol \lowercase{and} \c{S}ahin]{M\lowercase{ehmet} A\lowercase{kif} AKYOL, B\lowercase{ayram} \c{S}AH\.{I}N}
\address{Bing\"{o}l University, Deparment of Mathematics, 12000, Bing\"{o}l,
Turkey}
\email{makyol@bingol.edu.tr}
\address{In\"{o}n\"{u} university, Deparment of Mathematics, 44280, Malatya,
Turkey}
\email{bayram.sahin@inonu.edu.tr}

\begin{abstract}
As a generalization of semi-invariant submersions,
we introduce conformal semi-invariant submersions from almost Hermitian
manifolds onto Riemannian manifolds. We give examples, investigate the
geometry of foliations which are arisen from the definition of a conformal
submersion and show that there are certain product structures on the total space
of a conformal semi-invariant submersion. Moreover, we also check the
harmonicity of such submersions and find necessary and sufficient conditions
of a conformal semi-invariant submersion to be totally geodesic.
\end{abstract}

\maketitle

\section{Introduction}
One of the main method to compare two manifolds and transfer certain structures
from a manifold to another manifold is to define appropriate smooth maps
between them. Given two manifolds, if the rank of a differential map is
equal to the dimension of the source manifold, then such maps are called
immersions and if the rank of a differential map is equal to the target
manifold, then such maps are called submersions. Moreover, if these maps are
isometry between manifolds, then the immersion is called isometric immersion
(Riemannian submanifold) and the submersion is called Riemannian submersion.
Riemannian submersions between Riemannian manifolds were studied by O'Neill \cite{O}
and Gray \cite{Gray}, for recent developments on the geometry of Riemannian submanifolds
and Riemannian submersions, see: \cite{C} and \cite{FIP}, respectively.

The theory of submanifolds of K\"{a}hler manifolds is one of the
important branches of differential geometry. A submanifold of a
K\"{a}hler manifold is a complex (invariant) submanifold if the
tangent space of the submanifold at each point is invariant with
respect to the almost complex structure of the K\"{a}hler
manifold. Besides complex submanifolds of a K\"{a}hler manifold,
there is another important class of submanifolds called totally
real submanifolds. A totally real submanifold of a K\"{a}hler
manifold $\bar{M}$ is a submanifold of $\bar{M}$ such that the
almost complex structure $J$ of $\bar{M}$ carries the tangent
space of
the submanifold at each point into its normal space and the main properties of  such submanifolds established in \cite{Chen-Ogiue}, \cite{Ogiue} and  \cite{Yano-Kon1}. On the other hand, CR-submanifolds were defined by Bejancu \cite{Bejancu} as a
generalization of complex and totally real submanifolds. A
CR-submanifold is called proper if it is neither  complex nor
totally real submanifold.   Many authors
have studied above real submanifolds in various ambient manifolds and many interesting
results were obtained, see (\cite{C}, page: 322) for a survey on all these results.

As analogue of holomorphic submanifolds, holomorphic submersions were introduced by Watson \cite{Watson}
in seventies by using the notion of almost complex map.  We note that almost Hermitian submersions have been
extended to the almost contact manifolds \cite{Domingo}, locally
conformal K\"{a}hler manifolds \cite{Marrero-Rocha}, quaternion
K\"{a}hler manifolds
\cite{Ianus-Mazzocco-Vilcu} and Paraquaternionic manifold \cite{Candarella}, \cite{Vilcu}, see \cite{FIP} for holomorphic submersions and their extensions to the other manifolds. The main property of such
maps is that their vertical distributions are invariant with respect to almost complex map of total space. Therefore, the second author of the present paper
considered a new submersion defined on an almost Hermitian manifold such that the vertical distribution
is anti-invariant with respect to almost complex structure \cite{Sahin1}. He showed that such submersions have rich
geometric properties and they are useful for investigating the geometry of the total space. This new class of
submersions which is called anti-invariant submersions can be seen as an analogue of totally real submanifolds
in the submersion theory.
As a
generalization of holomorphic submersions and anti-invariant
submersions,  the second author introduced semi-invariant submersions from
almost Hermitian manifolds onto Riemannian manifolds and then he studied the geometry of such
maps in \cite{Sahin3}. We recall that a Riemannian submersion $F$ from an almost
Hermitian manifold$(M, J_{_M}, g_{_M})$ with almost complex
structure $J_{_M}$ to a Riemannian manifold $(N,g_{_N})$ is called a
semi-invariant submersion if the fibers have differentiable
distributions $D$ and $D^{\perp}$ such that  $D$ is invariant with
respect to $J_{_M}$ and its orthogonal complement $D^{\perp}$ is
totally real distribution, i.e, $J_{_M}(D^{\perp}_p)\subseteq (ker
F_*)^\perp.$ Obviously, almost Hermitian submersions \cite{Watson} and
anti-invariant submersions\cite{Sahin1} are semi-invariant submersions with $D^\perp=\{0\}$ and $D=\{0\}$, respectively. These new submersions have been studied in different total spaces, see:\cite{Erken-Murathan}, \cite{Lee}, \cite{Park1}, \cite{Park2}, \cite{Park3}, \cite{Fatima-Shahid}.

On the other hand, as a generalization of Riemannian submersions,
horizontally conformal submersions are defined as follows \cite{Baird-Wood}: Suppose
that $(M,g_{M})$ and $(B,g_{B})$ are Riemannian manifolds and $F:M\longrightarrow B$ is a smooth submersion, then $F$ is called a
horizontally conformal submersion, if there is a positive function $\lambda$ such that
\begin{equation*}
\lambda^{2}g_{M}(X,Y)=g_{B}(F_{*}X,F_{*}Y)
\end{equation*}
for every $X,Y\in\Gamma((kerF_{*})^\perp).$ It is obvious that every
Riemannian submersion is a particular horizontally conformal submersion with
$\lambda=1$. We note that horizontally conformal submersions are special
horizontally conformal maps which were introduced independently by Fuglede
\cite{F} and Ishihara \cite{I}. We also note that a horizontally conformal
submersion $F:M\longrightarrow B$ is said to be horizontally homothetic if
the gradient of its dilation $\lambda$ is vertical, i.e.,
\begin{equation}
\mathcal{H}(grad\lambda)=0 \label{e.q:1.1}
\end{equation}
at $p\in M$, where $\mathcal{H}$ is the projection on the horizontal space
$(kerF_{*})^{\perp}$.
One can see that Riemannian submersions are very special maps comparing with
conformal submersions. Although conformal maps does not preserve distance
between points contrary to isometries, they preserve angles between vector
fields. This property enables one to transfer certain properties of a
manifold to another manifold by deforming such properties.

As a generalization of holomorphic submersions, conformal holomorphic submersions were studi-ed by Gudmundsson and Wood \cite{GW}.
They obtained necessary and sufficient conditions for conformal holomorphic submersions to be a harmonic morphism, see also
\cite{Chi1}, \cite{Chi2} and \cite{Chi3} for the harmonicity of conformal holomorphic submersions. Moreover, in \cite{Akyol-Sahin}, we introduce conformal anti-invariant submersions, give examples and investigate the geometry of such submersions.

 In this paper, we study conformal semi-invariant submersions as a generalization of semi-invariant submersions and  investigate the geometry of the total space and the base space for the existence of such submersions.

The paper is organized as follows. In the second section, we gather main
notions and formulas for other sections. In section 3, we introduce conformal
semi-invariant submersions from almost Hermitian manifolds onto Riemannian
manifolds, give examples and investigate the geometry of leaves of the
horizontal distribution and the vertical distribution. In this section we also show that
there are certain product structures on the total space of a conformal
semi-invariant submersion. In section 4, we find necessary and sufficient
conditions for a conformal semi-invariant submersion to be harmonic and totally geodesic, respectively.

\section{Preliminaries}

In this section, we define almost Hermitian manifolds, recall the notion of
(horizontally) conformal submersions between Riemannian manifolds and give a
brief review of basic facts of (horizontally) conformal submersions. Let
$(M,g_{M})$ be an almost Hermitian manifold. This means \cite{YK} that $M$
admits a tensor field $J$ of type (1,1) on $M$ such that, $\forall X,Y\in\Gamma(TM)$, we have
\begin{equation}
J^{2}=-I,\,\,\, g_{M}(X,Y)=g_{M}(JX,JY). \label{e.q:2.1}
\end{equation}
An almost Hermitian manifold $M$ is called K\"{a}hler manifold if
\begin{equation}
(\nabla^{^M}_{X}J)Y=0,\mbox{ } \forall X,Y\in\Gamma(TM),  \label{e.q:2.2}
\end{equation}
where $\nabla^{^M}$ is the Levi-Civita connection on $M$. Conformal  submersions
belong to a wide class of conformal maps that we are going to recall their
definition, but we will not study such maps in this paper.

\begin{definition}\label{de1}
\textit{(\cite{Baird-Wood})} Let $\varphi:(M^{m},g)\longrightarrow (N^{n},h)$ be a
smooth map between Riemannian mani-folds, and let $x\in M$. Then $\varphi$ is
called horizontally weakly conformal or semi conformal at $x$ if either

\begin{enumerate}
\item[(i)] $d\varphi_{x}=0$, or
\item[(ii)] $d\varphi_{x}$ maps the horizontal space $\mathcal{H}%
_{x}=(ker(d\varphi_{x}))^\perp$ conformally onto $T_{\varphi_{*}}N$, i.e., $%
d\varphi_{x}$ is surjective and there exists a number $\Lambda(x)\neq0$ such
that
\begin{equation}
h(d\varphi_{x}X,d\varphi_{x}Y)=\Lambda(x)g(X,Y)\mbox{ }(X,Y\in\mathcal{H}%
_{x}).
\end{equation}
\end{enumerate}
\end{definition}
Note that we can write the last equation more succinctly as
\begin{equation*}
(\varphi^{*}h)_{x}\mid_{\mathcal{H}_{x}\times\mathcal{H}_{x}}=%
\Lambda(x)g_{x}\mid_{\mathcal{H}_{x}\times\mathcal{H}_{x}}.
\end{equation*}
With the above definition of critical point, a point $x$ is of type (i) in
Definition \ref{de1} if and only if it is a critical point of $\varphi$; we shall call
a point of type (ii) a \textit{regular point}. At a critical point, $%
d\varphi_{x}$ has rank $0$; at a regular point, $d\varphi_{x}$ has rank $n$
and $\varphi$ is submersion. The number $\Lambda(x)$ is called the \textit{%
square dilation} (of $\varphi$ at $x$); it is necessarily non-negative; its
square root $\lambda(x)=\sqrt{\Lambda(x)}$ is called the dilation (of $%
\varphi$ at $x$). The map $\varphi$ is called \textit{horizontally weakly
conformal} or \textit{semi conformal} (on $M$) if it is horizontally weakly
conformal at every point of $M$. It is clear that if $\varphi$ has no
critical points, then we call it a (\textit{horizontally}) conformal
submersion.

Next, we recall the following definition from \cite{Baird-Wood}. Let
$F:M\longrightarrow N$ be a submersion. A vector field $E$ on $M$ is said
to be projectable if there exists a vector field $\check{E}$ on $N$, such
that $F_*(E_{x})=\check{E}_{F(x)}$ for all $x\in M$. In this case $E$ and
$\check{E}$ are called $F-$ related. A horizontal vector field $Y$ on $%
(M,g)$ is called basic, if it is projectable. It is well known fact, that is
$\check{Z}$ is a vector field on $N$, then there exists a unique basic
vector field $Z$ on $M$, such that $Z$ and $\check{Z}$ are $F-$ related.
The vector field $Z$ is called the horizontal lift of $\check{Z}$.

The fundamental tensors of a submersion were introduced in \cite{O}. They
play a similar role to that of the second fundamental form of an immersion.
More precisely, O'neill's tensors $T$ and $A$ defined for vector fields $E,F$
on $M$ by
\begin{equation}  \label{A}
A_{E}F=\mathcal{V}\nabla^{^M}_{\mathcal{H}E}\mathcal{H}F +\mathcal{H}\nabla^{^M}_{%
\mathcal{H}E}\mathcal{V}F
\end{equation}
\begin{equation}  \label{T}
T_{E}F=\mathcal{H}\nabla^{^M}_{\mathcal{V}E}\mathcal{V}F +\mathcal{V}\nabla^{^M}_{%
\mathcal{V}E}\mathcal{H}F
\end{equation}
where $\mathcal{V}$ and $\mathcal{H}$ are the vertical and horizontal
projections (see \cite{FIP}). On the other hand, from (\ref{A}) and (\ref{T}%
), we have
\begin{equation}  \label{nvw}
\nabla^{^M}_{V}W=T_{V}W+\hat{\nabla}_{V}W
\end{equation}
\begin{equation}  \label{nvx}
\nabla^{^M}_{V}X=\mathcal{H}\nabla^{^M}_{V}X+T_{V}X
\end{equation}
\begin{equation}  \label{nxv}
\nabla^{^M}_{X}V=A_{X}V +\mathcal{V}\nabla^{^M}_{X}V
\end{equation}
\begin{equation}  \label{nxy}
\nabla^{^M}_{X}Y=\mathcal{H}\nabla^{^M}_{X}Y+A_{X}Y
\end{equation}
for $X,Y\in\Gamma((kerF_{*})^\perp)$ and $V,W\in\Gamma(kerF_{*})$, where $\hat{%
\nabla}_{V}W=\mathcal{V}\nabla^{^M}_{V}W$. If $X$ is basic, then $\mathcal{H}%
\nabla^{^M}_{V}X=A_{X}V$. It is easily seen that for $x\in M$, $X\in \mathcal{H}%
_{x}$ and $V\in\mathcal{V}_{x}$ the linear operators $T_{V}$, $%
A_{X}:T_{x}M\longrightarrow T_{x}M$ are skew-symmetric, that is
\begin{equation*}
-g(T_{V}E,G)=g(E,T_{V}G)\text{ and }-g(A_{X}E,G)=g(E,A_{X}G)
\end{equation*}
for all $E,G\in T_{x}M$. We also see that the restriction of $T$ to the
vertical distribution $T\mid_{V\times V}$ is exactly the second fundamental
form of the fibres of $F$. Since $T_{V}$ skew-symmetric we get: $F$ has
totally geodesic fibres if and only if $T\equiv0$.

We now recall the notion of harmonic maps between Riemannian manifolds. Let
$(M,g_{M})$ and $(N,g_{N})$ be Riemannian manifolds and suppose that $%
\varphi:M \longrightarrow N$ is a smooth map between them. Then the
differential of $\varphi_{*}$ of $\varphi$ can be viewed a section of the
bundle $Hom(TM,\varphi^{-1}TN) \longrightarrow M$, where $\varphi^{-1}TN$ is
the pullback bundle which has fibres $(\varphi^{-1}TN)_{p}\!\!=\!\!T_{\varphi(p)}N$,
$p\!\in\! M$. $Hom(TM,\varphi^{-1}TN)$ has a connection $\nabla$ induced from
the Levi-Civita connection $\nabla^{M}$ and the pullback connection. Then
the second fundamental form of $\varphi$ is given by
\begin{equation}
(\nabla\varphi_{*})(X,Y)=\nabla^{\varphi}_{X}\varphi_{*}(Y)-\varphi_{*}(%
\nabla^{M}_{X}Y) \label{nfixy}
\end{equation}
for $X,Y\in\Gamma(TM)$, where $\nabla^{\varphi}$ is the pullback connection.
It is known that the second fundamental form is symmetric. A smooth map $%
\varphi:(M,g_{M}) \longrightarrow (N,g_{N})$ is said to be harmonic if $%
trace(\nabla\varphi_{*})=0$. On the other hand, the tension field of $\varphi
$ is the section $\tau(\varphi)$ of $\Gamma(\varphi^{-1}TN)$ defined by
\begin{equation}
\tau(\varphi)=div\varphi_{*}=\sum_{i=1}^{m}(\nabla\varphi_{*})(e_{i},e_{i}),
\end{equation}
where $\{e_{1},...,e_{m}\}$ is the orthonormal frame on $M$. Then it follows
that $\varphi$ is harmonic if and only if $\tau(\varphi)=0$, for details,
see \cite{Baird-Wood}. Finally, we recall the following lemma from \cite{Baird-Wood}.

\begin{lemma}
\label{lem1} Suppose that $F:M\longrightarrow N$ is a horizontally conformal submersion. Then, for any
horizontal vector fields $X,Y$ and vertical fields $V,W$ we have
\begin{enumerate}
\item[(i)] $(\nabla F_*)(X,Y)=X(ln\lambda)F_*(Y)+Y(ln\lambda)F_*(X)-g(X,Y)F_*(gradln\lambda)$;
\item[(ii)] $(\nabla F_*)(V,W)=-F_*(T_{V}W)$;
\item[(iii)] $(\nabla F_*)(X,V)=-F_*(\nabla^{M}_{X}V)=-F_*(A_{X}V).$
\end{enumerate}
\end{lemma}

\section{Conformal Semi-invariant submersions}

In this section, we define conformal semi-invariant submersions from an almost
Hermitian manifold onto a Riemannian manifold, investigate the integrability
of distributions and show that there are certain product structures on the total
space of such submersions.

\begin{definition}
\label{def} Let $M$ be a complex $m$-dimensional almost Hermitian manifold
with Hermitian metric $g_{M}$ and almost complex structure $J$ and $N$ be a
Riemannian manifold with Riemannian metric $g_{N}.$ A horizontally conformal submersion
$F:M \longrightarrow N$ with dilation $\lambda$ is called conformal semi-invariant submersion if there
is a distribution $D_{1}\subseteq kerF_{*}$ such that
\begin{equation}
kerF_{*}=D_{1}\oplus D_{2}
\end{equation}
and
\begin{equation}
J(D_{1})=D_{1}, \mbox{} J(D_{2})\subseteq (kerF_{*})^\perp,
\end{equation}
where $D_{2}$ is orthogonal complementary to $D_{1}$ in $kerF_{*}$.
\end{definition}

We note that it is known that the distribution $kerF_{*}$ is integrable.
Hence, above definition implies that the integral manifold (fiber) $F^{-1}(q)
$, $q\in N$, of $kerF_{*}$ is CR-submanifold of $M$. For CR-submanifolds,
see \cite{Bejancu} and \cite{C}. We now give some examples of conformal semi-invariant submersions.
\begin{example}
Every semi-invariant submersion from an almost Hermitian manifold to
a Riemannian manifold is a conformal semi-invariant submersion with $\lambda=I$, where $I$ denotes the identity function.
\end{example}

We say that a conformal semi-invariant submersion is proper if $\lambda\neq I$.
We now present an example of a proper conformal semi-invariant submersion.
In the following $R^{2m}$ denotes the Euc-lidean $2m$-space with the standard
metric. An almost complex structure $J$ on $R^{2m}$ is said to be compatible
if $(R^{2m},J)$ is complex analytically isometric to the complex number space
$C^{m}$ with the standard flat K\"{a}hlerian metric. We denote by $J$ the
compatible almost complex structure on $R^{2m}$ defined by%
\begin{equation*}
J(a^{1},...,a^{2m})=(-a^{2},a^{1},...,-a^{2m},a^{2m-1}).
\end{equation*}
\begin{example}
\label{exm1}Let $F$ be a submersion defined by
\begin{equation*}
\begin{array}{cccc}
  F: & R^6             & \longrightarrow & R^2\\
     & (x_1,x_2,x_3,x_4,x_5,x_6) &             & (e^{x_{3}}\cos x_{5},e^{x_{3}}\sin x_{5}),
\end{array}
\end{equation*}
where $x_5\in\mathbb{R}-\{k\frac{\pi}{2},k\pi\},\ k\in\mathbb{R}$. Then it follows that
\begin{equation*}
kerF_{*}=span\{V_1=\partial x_1,\ V_2=\partial x_2,\ V_3=\partial x_4,\ V_4=\partial x_6\}
\end{equation*}
and
\begin{equation*}
(kerF_{*})^\perp=span\{X_1=e^{x_3}\cos{x_5}\partial x_3-e^{x_3}\sin{x_5}\partial x_5, \ X_2=e^{x_3}\sin{x_5}\partial x_3+e^{x_3}\cos{x_5}\partial x_5\}.
\end{equation*}
Hence we have $JV_1=V_2$, $JV_3=-e^{-x_3}\cos{x_5}X_1-e^{-x_3}\sin{x_5}X_2$ and $JV_4=e^{-x_3}\sin{x_5}X_1-e^{-x_3}\cos{x_5}X_2$. Thus it follows that
$D_1=span\{V_1,V_2\}$ and $D_2=span\{V_3,V_4\}$. Also by direct computations, we get
\begin{equation*}
F_{*}X_1=(e^{x_3})^{2}\partial y_1,\ F_{*}X_2=(e^{x_3})^{2}\partial y_2.
\end{equation*}
Hence, we have
\begin{equation*}
g_2(F_{*}X_1,F_{*}X_1)=(e^{x_3})^{2}g_1(X_1,X_1),\ \ g_2(F_{*}X_2,F_{*}X_2)=(e^{x_3})^{2}g_1(X_2,X_2),
\end{equation*}
where $g_1$ and $g_2$ denote the standard metrics (inner products) of $R^6$ and $R^2$.
Thus $F$ is a conformal semi-invariant submersion with $\lambda=e^{x_3}.$
\end{example}
We now investigate the integrability of the distributions $D_{1}$ and $D_{2}$.
\begin{lemma}
Let F be a conformal semi-invariant submersion from a K\"{a}hler
manifold $(M,g_{M},J)$ onto a Riemannian manifold $(N,g_{N})$. Then
\begin{enumerate}
\item[(i)] The distribution $D_{2}$ is always integrable.
\item[(ii)] The distribution $D_{1}$ integrable if and only if $(\nabla F_{*})(Y,JX)-(\nabla
F_{*})(X,JY)\in\Gamma(F_{*}(\mu))$
\end{enumerate}
for $X,Y\in\Gamma((kerF_{*})^\perp).$
\end{lemma}
\begin{proof}
Since the fibers of conformal semi-invariant submersions from K\"{a}hler manifolds are
CR-submanifolds and $T$ is the second fundamental form of the fibers,
$(i)$  can be deduced from Theorem 1.1 of [\cite{Bejancu},p.39].
(ii) We note that the distribution $D_{1}$ integrable if and only if $%
g_{M}([X,Y],Z)=g_{M}([X,Y],W)=0$ for $X,Y\in\Gamma(D_{1}), Z\in\Gamma(D_{2})$
and $W\in\Gamma((kerF_{*})^\perp)$. Since $kerF_{*}$ is integrable, we immediately have $%
g_{M}([X,Y],W)=0$. Thus $D_{1}$ is integrable if and only if $g_{M}([X,Y],Z)=0$. Since $F$ is a conformal submersion, by using (\ref{e.q:2.1}) and
Lemma \ref{lem1} we have
\begin{equation}
g_{M}([X,Y],Z)=\frac{1}{\lambda^{2}}g_{N}(F_{*}(\nabla^{^M}_{X}JY),F_{*}JZ)-\frac{1}{%
\lambda^{2}}g_{N}(F_{*}(\nabla^{^M}_{Y}JX),F_{*}JZ).\nonumber
\end{equation}
Then using (\ref{nfixy}) we get
\begin{equation}
g_{M}([X,Y],Z)=\frac{1}{\lambda^{2}}g_{N}((\nabla F_{*})(Y,JX)-(\nabla
F_{*})(X,JY),F_{*}JZ).\nonumber
\end{equation}
Thus proof is complete.
\end{proof}
Let $F$ be a conformal semi-invariant submersion from a K\"{a}hler manifold $%
(M,g_{M},J)$ onto a Riemannian manifold $(N,g_{N})$. We denote the
complementary distribution to $JD_{2}$ in $(kerF_{*})^\perp$ by $\mu$. Then
for $V\in\Gamma(kerF_{*})$, we write
\begin{equation}
JV=\phi V+\omega V  \label{e.q:3.3}
\end{equation}
where $\phi V\in\Gamma(D_{1})$ and $\omega V\in\Gamma (JD_{2})$. Also for $%
X\in\Gamma((kerF_{*})^\perp)$, we have
\begin{equation}
JX=\mathcal{B}X+\mathcal{C}X,  \label{e.q:3.4}
\end{equation}
where $\mathcal{B}X\in\Gamma(D_{2})$ and $\mathcal{C}X\in\Gamma(\mu)$. Then
by using (\ref{e.q:3.3}), (\ref{e.q:3.4}), (\ref{nvw}) and (\ref{nvx}) we
get
\begin{equation}
(\nabla^{^M}_{V}\phi)W=\mathcal{B}T_{V}W-T_{V}\omega W \label{e.q:3.5}
\end{equation}
\begin{equation}
(\nabla^{^M}_{V}\omega)W=\mathcal{C}T_{V}W-T_{V}\phi W \label{e.q:3.6}
\end{equation}
for $V,W\in\Gamma(kerF_{*})$, where
\begin{equation*}
(\nabla^{^M}_{V}\phi)W=\hat{\nabla}_{V}\phi W-\phi\hat{\nabla}_{V}W
\end{equation*}
and
\begin{equation*}
(\nabla^{^M}_{V}\omega)W=\mathcal{H}\nabla^{^M}_{V}\omega W-\omega\hat{\nabla}_{V}W.
\end{equation*}

We now study the integrability of the distribution $(kerF_{*})^\perp$ and
then we investigate the geometry of leaves of $kerF_{*}$ and $%
(kerF_{*})^\perp$.
\begin{theorem}\label{teo1}
Let $F$ be a conformal semi-invariant submersion from a K\"{a}hler manifold $(M,g_{M},J)$ to a Riemannian manifold $(N,g_{N})$. Then the distribution $(kerF_{*})^\perp$ is integrable if and only if
\begin{equation*}
A_{Y}\omega\mathcal{B}X-A_{X}\omega\mathcal{B}Y+J(A_{Y}\mathcal{C}X-A_{X}\mathcal{C}Y)\notin\Gamma(D_1)
\end{equation*}
and
\begin{eqnarray*}
\lambda^{2}g_{N}(\nabla^{F}_{Y}F_{*}\mathcal{C}X-\nabla^{F}_{X}F_{*}\mathcal{C}Y,F_{*}JW)&=&g_{M}(A_{Y}\mathcal{B}X-A_{X}\mathcal{B}Y-\mathcal{C}Y(ln\lambda)X+\mathcal{C}X(ln\lambda)Y\\
&+&2g_{M}(X,\mathcal{C}Y)gradln\lambda,JW)
\end{eqnarray*}
for $X,Y\in\Gamma((kerF_{*})^\perp)$, $V\in\Gamma(D_1)$ and $W\in\Gamma(D_2)$.
\end{theorem}
\begin{proof}
           The distribution $(kerF_{*})^\perp$ is integrable on $M$ if and only if
           $$g_{M}([X,Y],V)=0\quad \mathrm{and}\quad g_{M}([X,Y],W)=0$$
            for $X,Y\in\Gamma((kerF_{*})^\perp)$, $V\in\Gamma(D_1)$ and $W\in\Gamma(D_2)$. Using (\ref{e.q:2.1}), (\ref{e.q:2.2}) and (\ref{e.q:3.4}), we get
\begin{align*}
g_{M}([X,Y],V)=g_{M}(\nabla^{^M}_{X}\mathcal{B}Y,JV)+g_{M}(\nabla^{^M}_{X}\mathcal{C}Y,JV)
-g_{M}(\nabla^{^M}_{Y}\mathcal{B}X,JV)-g_{M}(\nabla^{^M}_{Y}\mathcal{C}X,JV).
\end{align*}
Also  using (\ref{nxy}), we get
\begin{align*}
g_{M}([X,Y],V)&=-g_{M}(\mathcal{B}Y,\nabla^{^M}_{X}JV)+g_{M}(A_{X}\mathcal{C}Y,JV)+g_{M}(\mathcal{B}X,\nabla^{^M}_{Y}JV)-g_{M}(A_{Y}\mathcal{C}X,JV).
\end{align*}
From (\ref{e.q:2.2}), (\ref{nxv}) and (\ref{e.q:3.3}), we derive
\begin{align}\label{e.q:3.7}
g_{M}([X,Y],V)&=g_{M}(A_{Y}\omega\mathcal{B}X-A_{X}\omega\mathcal{B}Y-JA_{X}\mathcal{C}Y+JA_{Y}\mathcal{C}X,V).
\end{align}
On the other hand, from (\ref{e.q:2.1}), (\ref{e.q:2.2}) and (\ref{e.q:3.4}) we derive
\begin{align*}
g_{M}([X,Y],W)&=g_{M}(\nabla^{^M}_{X}\mathcal{B}Y,JW)+g_{M}(\nabla^{^M}_{X}\mathcal{C}Y,JW)
-g_{M}(\nabla^{^M}_{Y}\mathcal{B}X,JW)-g_{M}(\nabla^{^M}_{Y}\mathcal{C}X,JW).
\end{align*}
Since $F$ is a conformal submersion, using (\ref{nfixy}) and Lemma \ref{lem1} we arrive at
\begin{align*}
g_{M}([X,Y],W)&=-\frac{1}{\lambda^2}g_{N}((\nabla F_{*})(X,\mathcal{B}Y),F_{*}JW)+\frac{1}{\lambda^2}g_{N}((\nabla F_{*})(Y,\mathcal{B}X),F_{*}JW)\\
&+\frac{1}{\lambda^2}g_{N}\{-X(ln\lambda)F_{*}\mathcal{C}Y-\mathcal{C}Y(ln\lambda)F_{*}X+g_{M}(X,\mathcal{C}Y)F_{*}(gradln\lambda)
+\nabla^{F}_{X}F_{*}\mathcal{C}Y,F_{*}JW\}\\
&-\frac{1}{\lambda^2}g_{N}\{-Y(ln\lambda)F_{*}\mathcal{C}X-\mathcal{C}X(ln\lambda)F_{*}Y+g_{M}(Y,\mathcal{C}X)F_{*}(gradln\lambda)
+\nabla^{F}_{Y}F_{*}\mathcal{C}X,F_{*}JW\}.
\end{align*}
Thus from (\ref{nfixy}) and (\ref{nxv}) we have
\begin{align*}
g_{M}([X,Y],W)&=-g_{M}(A_{X}\mathcal{B}Y,JW)+g_{M}(A_{Y}\mathcal{B}X,JW)-\frac{1}{\lambda^2}g_{M}(gradln\lambda,X)g_{N}(F_{*}\mathcal{C}Y,F_{*}JW)\\
&-\frac{1}{\lambda^2}g_{M}(gradln\lambda,\mathcal{C}Y)g_{N}(F_{*}X,F_{*}JW)+\frac{1}{\lambda^2}g_{M}(X,\mathcal{C}Y)g_{N}(F_{*}(gradln\lambda),F_{*}JW)\\
&+\frac{1}{\lambda^2}g_{N}(\nabla^{F}_{X}F_{*}\mathcal{C}Y,F_{*}JW)+\frac{1}{\lambda^2}g_{M}(gradln\lambda,Y)g_{N}(F_{*}\mathcal{C}X,F_{*}JW)\\
&+\frac{1}{\lambda^2}g_{M}(gradln\lambda,\mathcal{C}X)g_{N}(F_{*}Y,F_{*}JW)-\frac{1}{\lambda^2}g_{M}(Y,\mathcal{C}X)g_{N}(F_{*}(gradln\lambda),F_{*}JW)\\
&-\frac{1}{\lambda^2}g_{N}(\nabla^{F}_{Y}F_{*}\mathcal{C}X,F_{*}JW).
\end{align*}
Moreover, using Definition \ref{def}, we obtain
\begin{align}\label{e.q:3.8}
g_{M}([X,Y],W)&=g_{M}(A_{Y}\mathcal{B}X-A_{X}\mathcal{B}Y-\mathcal{C}Y(ln\lambda)X+\mathcal{C}X(ln\lambda)Y+2g_{M}(X,\mathcal{C}Y)gradln\lambda,JW)\notag \\
&-\frac{1}{\lambda^2}g_{N}(\nabla^{F}_{Y}F_{*}\mathcal{C}X-\nabla^{F}_{X}F_{*}\mathcal{C}Y,F_{*}JW).
\end{align}
Thus proof follows from (\ref{e.q:3.7}) and (\ref{e.q:3.8}).
\end{proof}
Next theorem gives a necessary and sufficient condition for conformal submersion to be a homothetic map.
\begin{theorem}
Let $F$ be a conformal semi-invariant submersion  from a K\"{a}hler manifold $(M,g_{M},J)$ to a Riemannian manifold $(N,g_{N})$ with integrable distribution $(kerF_{*})^\perp$. Then $F$ is a horizontally homothetic map if and only if
\begin{equation}\label{e.q:3.9}
\lambda^{2}g_{M}(A_{Y}\mathcal{B}X-A_{X}\mathcal{B}Y,JW)=g_{N}(\nabla^{F}_{Y}F_{*}\mathcal{C}X-\nabla^{F}_{X}F_{*}\mathcal{C}Y,F_{*}JW)
\end{equation}
for $X,Y\in\Gamma((kerF_{*})^\perp)$ and $W\in\Gamma(D_2)$.
\end{theorem}
\begin{proof}
For $X,Y\in\Gamma((kerF_{*})^\perp)$ and $W\in\Gamma(D_2)$, from (\ref{e.q:3.8}) we have
\begin{align*}
g_{M}([X,Y],W)&=g_{M}(A_{Y}\mathcal{B}X-A_{X}\mathcal{B}Y-\mathcal{C}Y(ln\lambda)X+\mathcal{C}X(ln\lambda)Y+2g_{M}(X,\mathcal{C}Y)gradln\lambda,JW)\\
&-\frac{1}{\lambda^2}g_{N}(\nabla^{F}_{Y}F_{*}\mathcal{C}X-\nabla^{F}_{X}F_{*}\mathcal{C}Y,F_{*}JW).
\end{align*}
If $F$ is a horizontally homothetic map then we get (\ref{e.q:3.9}). Conversely, if (\ref{e.q:3.9}) is satisfied then we get
\begin{equation}\label{e.q:3.10}
g_{M}(-g_{M}(gradln\lambda,\mathcal{C}Y)X+g_{M}(gradln\lambda,\mathcal{C}X)Y+2g_{M}(X,\mathcal{C}Y)gradln\lambda,JW)=0.\end{equation}
Now, taking $Y=JW$ for $W\in\Gamma(D_2)$ in (\ref{e.q:3.10}), we have $g_{M}(gradln\lambda,\mathcal{C}X)g_{M}(JW,JW)=0.$ Thus, $\lambda$ is a constant on $\Gamma(\mu).$ On the other hand, taking $Y=\mathcal{C}X$ for $X\in\Gamma(\mu)$ in (\ref{e.q:3.10}) we obtain
\begin{equation*}
2g_{M}(X,\mathcal{C}^{2}X)g_{M}(gradln\lambda,JW)=2g_{M}(X,X)g_{M}(gradln\lambda,JW)=0.
\end{equation*}
From above equation, $\lambda$ is a constant on $\Gamma(JD_2).$ This completes the proof.
\end{proof}
As conformal version of anti-holomorphic semi-invariant submersion (\cite{T}), a conformal semi-invariant submersion is called a conformal anti-holomorphic semi-invariant submersion if $J(D_2)=(kerF_{*})^\perp$ . For a conformal anti-holomorphic semi-invariant submersion, from Theorem \ref{teo1} we have the following.
\begin{corollary}
Let $F$ be a conformal anti-holomorphic semi-invariant submersion from a K\"{a}hler manifold $(M,g_{M},J)$ to a Riemannian manifold $(N,g_{N})$. Then the following assertions are equivalent to each other;
\begin{enumerate}
\item [(i)] $(kerF_{*})^\perp$ is integrable
\item [(ii)] $g_{N}(F_{*}JW_1,(\nabla F_{*})(V,JW_2))=g_{N}(F_{*}JW_2,(\nabla F_{*})(V,JW_1))$
\end{enumerate}
for $W_1,W_2\in\Gamma(D_2)$ and $V\in\Gamma(kerF_{*})$.
\end{corollary}
For the geometry of leaves of the horizontal distribution, we have the following theorem.
\begin{theorem}\label{teo2}
Let $F$ be a conformal semi-invariant submersion from a K\"{a}hler manifold $(M,g_{M},J)$ to a Riemannian manifold $(N,g_{N})$. Then the distribution $(kerF_{*})^\perp$ defines a totally geodesic foliation on $M$ if and only if
\begin{equation*}
A_{X}\mathcal{C}Y+\mathcal{V}\nabla^{^M}_{X}\mathcal{B}Y\in\Gamma(D_2)
\end{equation*}
and
\begin{equation*}
\lambda^{2}\{g_{M}(A_{X}\mathcal{B}Y-\mathcal{C}Y(ln\lambda)X+g_{M}(X,\mathcal{C}Y)gradln\lambda,JW)\}=g_{N}(\nabla^{F}_{X}F_{*}JW,F_{*}\mathcal{C}Y)
\end{equation*}
for $X,Y\in\Gamma((kerF_{*})^\perp)$, $V\in\Gamma(D_1)$ and $W\in\Gamma(D_2)$.
\end{theorem}
\begin{proof}
           The distribution $(kerF_{*})^\perp$ defines a totally geodesic foliation on $M$ if and only if $g_{M}(\nabla^{^M}_{X}Y,V)=0$ and $g_{M}(\nabla^{^M}_{X}Y,W)=0$ for $X,Y\in\Gamma((kerF_{*})^\perp)$, $V\in\Gamma(D_1)$ and $W\in\Gamma(D_2)$.
Then by using (\ref{e.q:2.1}), (\ref{e.q:2.2}),(\ref{nxv}), (\ref{nxy}) and (\ref{e.q:3.4}), we get
\begin{align}\label{e.q:3.11}
g_{M}(\nabla^{^M}_{X}Y,V)&=-g_{M}(\phi(A_{X}\mathcal{C}Y+\mathcal{V}\nabla^{^M}_{X}\mathcal{B}Y),V).
\end{align}
On the other hand, from (\ref{e.q:2.1}), (\ref{e.q:2.2}) and (\ref{e.q:3.4}) we get
\begin{align*}
g_{M}(\nabla^{^M}_{X}Y,W)&=-g_{M}(\mathcal{B}Y,\nabla^{^M}_{X}JW)-g_{M}(\mathcal{C}Y,\nabla^{^M}_{X}JW).
\end{align*}
Since $F$ is a conformal submersion, using (\ref{nxy}), (\ref{nfixy}) and Lemma \ref{lem1} we arrive at
\begin{align*}
g_{M}(\nabla^{^M}_{X}Y,W)&=-g_{M}(\mathcal{B}Y,A_{X}JW)++\frac{1}{\lambda^2}g_{M}(gradln\lambda,JW)g_{N}(F_{*}X,F_{*}\mathcal{C}Y)\\
&-\frac{1}{\lambda^2}g_{M}(X,JW)g_{N}(F_{*}(gradln\lambda),F_{*}\mathcal{C}Y)
-\frac{1}{\lambda^2}g_{N}(\nabla^{F}_{X}F_{*}JW,F_{*}\mathcal{C}Y).
\end{align*}
Conformal semi-invariant  $F$ implies that
\begin{align}\label{e.q:3.12}
g_{M}(\nabla^{^M}_{X}Y,V)&=g_{M}(A_{X}\mathcal{B}Y-\mathcal{C}Y(ln\lambda)X+g_{M}(X,\mathcal{C}Y)gradln\lambda,JW)\notag\\
&-\frac{1}{\lambda^2}g_{N}(\nabla^{F}_{X}F_{*}JW,F_{*}\mathcal{C}Y).
\end{align}
Thus proof follows from (\ref{e.q:3.11}) and (\ref{e.q:3.12}).
\end{proof}
Next we give new conditions for conformal semi-invariant submersions to be horizontally homot-hetic map. But we first give the following definition.
\begin{definition}
Let $F$ be a conformal semi-invariant submersion from a K\"{a}hler manifold $(M,g_{M},J)$ to a Riemannian manifold $(N,g_{N})$. Then we say that $D_2$ is parallel along $(kerF_{*})^\perp$ if $\nabla^{^M}_{X}W\in\Gamma(D_2)$ for $X\in\Gamma((kerF_{*})^\perp)$ and $W\in\Gamma(D_2).$
\end{definition}
\begin{corollary}
Let $F:(M,g_{M},J)\longrightarrow (N,g_{N})$ be a conformal semi-invariant submersion such that $D_2$ is parallel along $(kerF_{*})^\perp$. Then $F$ is a horizontally homothetic map if and only if
\begin{equation}\label{e.q:3.13}
\lambda^{2}g_{M}(A_{X}\mathcal{B}Y,JW)=g_{N}(\nabla^{F}_{X}F_{*}JW,F_{*}\mathcal{C}Y)
\end{equation}
for $X,Y\in\Gamma((kerF_{*})^\perp)$ and $W\in\Gamma(D_2)$.
\end{corollary}
\begin{proof}
 (\ref{e.q:3.13}) implies that
\begin{equation}\label{e.q:3.14}
-g_{M}(gradln\lambda,\mathcal{C}Y)g_{M}(X,JW)+g_{M}(X,\mathcal{C}Y)g_{M}(gradln\lambda,JW)=0.
\end{equation}
Now, taking $X=JW$ for $W\in\Gamma(D_{2})$ in (\ref{e.q:3.14}), we get
$$g_{M}(gradln\lambda,\mathcal{C}Y)g_{M}(JW,JW)=0.$$ Thus, $\lambda$ is a constant on $\Gamma(\mu)$. On the other hand, taking $X=CY$ for $Y\in\Gamma(\mu)$ in (\ref{e.q:3.14}) we derive $$g_{M}(\mathcal{C}Y,\mathcal{C}Y)g_{M}(gradln\lambda,JW)=0.$$ From above equation, $\lambda$ is a constant on $\Gamma(JD_{2})$. The converse is clear from (\ref{e.q:3.12}).
\end{proof}
In particular, if $F$ is a conformal anti-holomorphic semi-invariant submersion, then we have the following.
\begin{corollary}\label{cor1}
Let $F$ be a conformal anti-holomorphic semi-invariant submersion from a K\"{a}hler manifold $(M,g_{M},J)$ to a Riemannian manifold $(N,g_{N})$. Then the following assertions are equivalent to each other;
\begin{enumerate}
\item [(i)] $(kerF_{*})^\perp$ defines a totally geodesic foliation on $M$.
\item [(ii)] $(\nabla F_{*})(V,JW_1)\in\Gamma(F_{*}(\mu))$
\end{enumerate}
for $W_1,W_2\in\Gamma(D_2)$ and $V\in\Gamma(kerF_{*})$.
\end{corollary}
In the sequel we are going to investigate the geometry of leaves of the distribution $kerF_{*}$.
\begin{theorem}\label{teo3}
Let $F$ be a conformal semi-invariant submersion from a K\"{a}hler manifold $(M,g_{M},J)$ to a Riemannian manifold $(N,g_{N})$. Then the distribution $(kerF_{*})$ defines a totally geodesic foliation on $M$ if and only if
\begin{equation*}
\lambda^{2}\{g_{M}(\mathcal{C}T_{U}\phi V+A_{\omega V}\phi U+g_{M}(\omega V,\omega U)gradln\lambda,X)\}
=g_{N}(\nabla^{F}_{\omega V}F_{*}X,F_{*}\omega U)
\end{equation*}
and
\begin{equation*}
T_{V}\omega U+\hat{\nabla}_{V}\phi U\in\Gamma(D_1)
\end{equation*}
for $U,V\in\Gamma(kerF_{*})$, $X\in\Gamma(\mu)$ and $W\in\Gamma(D_2)$.
\end{theorem}
\begin{proof}
The distribution $(kerF_{*})$ defines a totally geodesic foliation on $M$ if and only if $g_{M}(\nabla^{^M}_{U}V,X)=0$ and $g_{M}(\nabla^{^M}_{U}V,JW)=0$ for $U,V\in\Gamma(kerF_{*})$, $X\in\Gamma(\mu)$ and $W\in\Gamma(D_2)$. Using (\ref{e.q:2.1}), (\ref{e.q:2.2}) and (\ref{e.q:3.3}), we have
$$
g_{M}(\nabla^{^M}_{U}V,X)=g_{M}(\nabla^{^M}_{U}\phi V,JX)+g_{M}(\phi U,\nabla^{^M}_{\omega V}X)+g_{M}(\omega U,\nabla^{^M}_{\omega V}X).
$$
Since $F$ is a conformal submersion, from (\ref{nvw}), (\ref{nxy}), (\ref{nfixy}) and Lemma \ref{lem1} we arrive at
\begin{align*}
g_{M}(\nabla^{^M}_{U}V,X)&=g_{M}(T_{U}\phi V,JX)+g_{M}(\phi U,A_{\omega V}X)-\frac{1}{\lambda^2}g_{M}(gradln\lambda,X)g_{N}(F_{*}\omega V,F_{*}\omega U)\\
& +\frac{1}{\lambda^2}g_{N}(\nabla^{F}_{\omega V}F_{*}X,F_{*}\omega U).
\end{align*}
Hence, we obtain
\begin{align}\label{e.q:3.15}
g_{M}(\nabla^{^M}_{U}V,X)&=g_{M}(-\mathcal{C}T_{U}\phi V-A_{\omega V}\phi U-g_{M}(\omega V,\omega U)gradln\lambda,X)\notag\\
&+\frac{1}{\lambda^2}g_{N}(\nabla^{F}_{\omega V}F_{*}X,F_{*}\omega U).
\end{align}
On the other hand, by using (\ref{e.q:2.1}), (\ref{e.q:2.2}) and (\ref{e.q:3.3}), we get
\begin{align}\label{e.q:3.16}
g_{M}(\nabla^{^M}_{U}V,JW)=-g_{M}(\omega(\hat{\nabla}_{V}\phi U+T_{V}\omega U),JW).
\end{align}
Thus proof follows from (\ref{e.q:3.15}) and (\ref{e.q:3.16}).
\end{proof}
Next we give certain conditions for dilation $\lambda$ to be constant on $\mu$. We first give the following definition.
\begin{definition}
Let $F$ be a conformal semi-invariant submersion from a K\"{a}hler manifold $(M,g_{M},J)$ to a Riemannian manifold $(N,g_{N})$. Then we say that $\mu$ is parallel along $kerF_{*}$ if $\nabla^{^M}_{U}X\in\Gamma(\mu)$ for $X\in\Gamma(\mu)$ and $U\in\Gamma(kerF_{*}).$
\end{definition}
\begin{corollary}
Let $F:(M,g_{M},J)\longrightarrow (N,g_{N})$ be a conformal semi-invariant submersion such that $\mu$ is parallel along $(kerF_{*})$. Then $F$ is a constant on $\mu$  if and only if
\begin{equation}\label{e.q:3.17}
\lambda^{2}g_{M}(\mathcal{C}T_{U}\phi V+A_{\omega V}\phi U,X)=g_{N}(\nabla^{F}_{\omega V}F_{*}X,F_{*}\omega U)
\end{equation}
for $U,V\in\Gamma(kerF_{*})$ and $X\in\Gamma(\mu)$.
\end{corollary}
\begin{proof}
If we have (\ref{e.q:3.17}) then we have
\begin{equation}
g_{M}(\omega V,\omega U)g_{M}(gradln\lambda,X)=0.
\end{equation}
From above equation, $\lambda$ is a constant on $\Gamma(\mu)$. The converse comes from  (\ref{e.q:3.15})
\end{proof}
From Theorem \ref{teo2} and Theorem \ref{teo3}, we have the following decomposition for total space;
\begin{theorem}
Let $F:(M,g_{M},J)\longrightarrow (N,g_{N})$ be a conformal semi-invariant submersion, where $(M,g_{M},J)$ is a a K\"{a}hler manifold $(N,g_{N})$ is a Riemannian manifold. Then $M$ is a locally product manifold of the form $M_{(kerF_{*})}\times_{\lambda} M_{(kerF_{*})^\perp}$ if and only if
\begin{equation*}
\lambda^{2}\{g_{M}(\mathcal{C}T_{U}\phi V+A_{\omega V}\phi U+g_{M}(\omega V,\omega U)gradln\lambda,X)\}
=g_{N}(\nabla^{F}_{\omega V}F_{*}X,F_{*}\omega U),
\end{equation*}
\begin{equation*}
T_{V}\omega U+\hat{\nabla}_{V}\phi U\in\Gamma(D_1)
\end{equation*}
and
\begin{equation*}
A_{X}\mathcal{C}Y+\mathcal{V}\nabla^{^M}_{X}\mathcal{B}Y\in\Gamma(D_2),
\end{equation*}
\begin{equation*}
\lambda^{2}\{g_{M}(A_{X}\mathcal{B}Y-\mathcal{C}Y(ln\lambda)X+g_{M}(X,\mathcal{C}Y)gradln\lambda,JW)\}=g_{N}(\nabla^{F}_{X}F_{*}JW,F_{*}\mathcal{C}Y)
\end{equation*}
for $X,Y\in\Gamma((kerF_{*})^\perp)$ and $U,V,W\in\Gamma(kerF_{*})$.
\end{theorem}
Since $(kerF_{*})^\perp=J(D_{2})\oplus \mu$ and $F$ is a conformal semi-invariant submersion from an almost Hermitian manifold $(M,g_{M},J)$ to a Riemannian manifold $(N,g_{N})$, for $X\in\Gamma(D_{2})$ and $Y\in\Gamma(\mu)$, we have $$\frac{1}{\lambda^2}g_{N}(F_{*}JX,F_{*}Y)=g_{M}(JX,Y)=0.$$ This implies that the distributions $F_{*}(JD_{2})$ and $F_{*}(\mu)$ are orthogonal.
Now, we investigate the geometry of the leaves of the distributitons $D_{1}$ and $D_{2}$.
\begin{theorem}\label{teo4}
Let $F$ be a conformal semi-invariant submersion from a K\"{a}hler manifold $(M,g_{M},J)$ to a Riemannian manifold $(N,g_{N})$. Then $D_{1}$ defines a totally geodesic foliation on $M$ if and only if $$(\nabla F_{*})(X_{1},JY_{1})\in\Gamma(F_*\mu)$$ and $$\frac{1}{\lambda^2}g_{N}((\nabla F_{*})(X_{1},JY_{1}),F_{*}\mathcal{C}X)=g_{M}(Y_{1},T_{X_1}\omega\mathcal{B}X)$$ for $X_{1},Y_{1}\in\Gamma(D_{1})$ and $X\in\Gamma((kerF_{*})^\perp)$.
\end{theorem}
\begin{proof}
The distribution $D_{1}$ defines a totally geodesic foliation on $M$ if and only if $g_{M}(\nabla^{^M}_{X_{1}}Y_{1},X_{2})=0$ and $g_{M}(\nabla^{^M}_{X_{1}}Y_{1},X)=0$ for $X_{1},Y_{1}\in\Gamma(D_{1}),X_{2}\in\Gamma(D_{2})$ and $X\in\Gamma((kerF_{*})^\perp)$. Using (\ref{e.q:2.1}) and (\ref{e.q:2.2}), we get
$$
g_{M}(\nabla^{^M}_{X_{1}}Y_{1},X_{2})=g_{M}(\mathcal{H}\nabla^{^M}_{X_{1}}JY_{1},JX_{2}).$$
Since $F$ is a conformal semi-invariant submersion, using (\ref{nfixy}) we have
\begin{align}\label{nx1y1x2}
g_{M}(\nabla^{^M}_{X_{1}}Y_{1},X_{2})&=-\frac{1}{\lambda^2}g_{N}((\nabla F_{*})(X_{1},JY_{1}),F_{*}JX_{2}).
\end{align}
On the other hand, by using (\ref{e.q:2.1}), (\ref{e.q:2.2}), (\ref{nvw}) and (\ref{e.q:3.4}) we derive
$$
g_{M}(\nabla^{^M}_{X_{1}}Y_{1},X)=g_{M}(Y_1,\nabla^{^M}_{X_1}J\mathcal{B}X)+g_{M}(\mathcal{H}\nabla^{^M}_{X_{1}}JY_{1},\mathcal{C}X).
$$
Now from (\ref{nvx}), (\ref{nfixy}) and (\ref{e.q:3.3}) we have
\begin{equation}\label{x1y1x}
g_{M}(\nabla^{^M}_{X_{1}}Y_{1},X)=g_{M}(Y_1,T_{X_1}\omega\mathcal{B}X)-\frac{1}{\lambda^2}g_{N}((\nabla F_{*})(X_{1},JY_{1}),F_{*}\mathcal{C}X).
\end{equation}
Thus proof follows from (\ref{nx1y1x2}) and (\ref{x1y1x}).
\end{proof}
For $D_{2}$ we have the following result.
\begin{theorem}\label{teo5}
Let $F$ be a conformal semi-invariant submersion from a K\"{a}hler manifold $(M,g_{M},J)$ to a Riemannian manifold $(N,g_{N})$. Then $D_{2}$ defines a totally geodesic foliation on $M$ if and only if $$(\nabla F_{*})(X_{2},JX_{1})\in\Gamma(F_*\mu)$$ and $$-\frac{1}{\lambda^2}g_{N}(\nabla^{F}_{JY_{2}} F_{*}JX_{2},F_{*}J\mathcal{C}X)=g_{M}(Y_{2},\mathcal{B}T_{X_{2}}\mathcal{B}X)+g_{M}(X_{2},Y_{2})g_{M}(\mathcal{H}gradln\lambda,J\mathcal{C}X)$$ for $X_{2},Y_{2}\in\Gamma(D_{2}),X_{1}\in\Gamma(D_{1})$ and $X\in\Gamma((kerF_{*})^\perp)$.
\end{theorem}
\begin{proof}
The distribution $D_{2}$ defines a totally geodesic foliation on $M$ if and only if $g_{M}(\nabla^{^M}_{X_{2}}Y_{2},X_{1})=0$ and $g_{M}(\nabla^{^M}_{X_{2}}Y_{2},X)=0$ for $X_{2},Y_{2}\in\Gamma(D_{2}),X_{1}\in\Gamma(D_{1})$ and $X\in\Gamma((kerF_{*})^\perp)$. Since $F$ is conformal semi-invariant submersion, using (\ref{e.q:2.1}) and (\ref{e.q:2.2}) and (\ref{nfixy}), we get
\begin{align}\label{nx2y2x1}
g_{M}(\nabla^{^M}_{X_{2}}Y_{2},X_{1})&=\frac{1}{\lambda^2}g_{N}((\nabla F_{*})(X_{2},JX_{1}),F_{*}JY_{2}).
\end{align}
On the other hand, by using (\ref{e.q:2.1}), (\ref{e.q:2.2}), (\ref{nvw}) and (\ref{e.q:3.4}) we derive
$$
g_{M}(\nabla^{^M}_{X_{2}}Y_{2},X)=-g_{M}(JY_{2},T_{X_{2}}\mathcal{B}X)+g_{M}([X_{2},JY_{2}]+\nabla^{^M}_{JY_{2}}X_{2},\mathcal{C}X).$$
Since $[X_{2},JY_{2}]\in \Gamma(ker F_*)$, we obtain
$$
 g_{M}(\nabla^{^M}_{X_{2}}Y_{2},X)=-g_{M}(JY_{2},T_{X_{2}}\mathcal{B}X)+g_{M}(\nabla^{^M}_{JY_{2}}JX_{2},J\mathcal{C}X).
$$
Then conformal semi-invariant submersion $F$, (\ref{e.q:3.4}) and  (\ref{nfixy}) imply
\begin{align}\label{nx2y2x}
g_{M}(\nabla^{^M}_{X_{2}}Y_{2},X)&=g_{M}(Y_{2},\mathcal{B}T_{X_{2}}\mathcal{B}X)+g_{M}(Y_{2},X_{2})g_{M}(\mathcal{H}gradln\lambda,J\mathcal{C}X)\notag\\
&+\frac{1}{\lambda^2}g_{N}(\nabla^{F}_{JY_{2}}F_{*}JX_{2},F_{*}J\mathcal{C}X).
\end{align}
Thus proof follows from (\ref{nx2y2x1}) and (\ref{nx2y2x}).
\end{proof}
From Theorem \ref{teo4} and Theorem \ref{teo5}, we have the following theorem;
\begin{theorem}
Let $F:(M,g_{M},J)\longrightarrow (N,g_{N})$ be a conformal semi-invariant submersion from a K\"{a}hler manifold  $(M,g_{M},J)$  onto  a Riemannian manifold $(N,g_{N})$.  Then the fibers of $F$ are locally product manifold if and only if
\begin{equation*}
(\nabla F_{*})(X_{1},JY_{1})\in\Gamma(F_*\mu),
\end{equation*}
\begin{equation*}
\frac{1}{\lambda^2}g_{N}((\nabla F_{*})(X_{1},JY_{1}),F_{*}\mathcal{C}X)=g_{M}(Y_{1},T_{X_1}\omega\mathcal{B}X)
\end{equation*}
and
\begin{equation*}
(\nabla F_{*})(X_{2},JX_{1})\in\Gamma(F_*\mu),
\end{equation*}
\begin{equation*}
-\frac{1}{\lambda^2}g_{N}(\nabla^{F}_{JY_{2}} F_{*}JX_{2},F_{*}J\mathcal{C}X)=g_{M}(Y_{2},\mathcal{B}T_{X_{2}}\mathcal{B}X)+g_{M}(X_{2},Y_{2})g_{M}(\mathcal{H}gradln\lambda,J\mathcal{C}X)
\end{equation*}
for any $X_{1},Y_{1}\in\Gamma(D_{1})$, $X_{2},Y_{2}\in\Gamma(D_{2})$ and $X\in\Gamma((kerF_{*})^\perp)$.
\end{theorem}
\section{Harmonicity of Conformal Semi-invariant submersions}
In this section, we are going to find necessary and sufficient conditions for a conformal semi-invariant submersions to be harmonic. We also investigate the necessary and sufficient conditions for such submersions to be totally geodesic. By considering the decomposition of the total space of conformal semi-invariant submersion, the following lemma comes from \cite{Baird-Wood}.
\begin{lemma}\label{teo6}
Let $F:(M^{2(m+n+r)},g_{M},J)\longrightarrow (N^{n+2r},g_{N})$ be a conformal semi-invariant submersion, where $(M,g_{M},J)$ is a K\"{a}hler manifold and $(N,g_{N})$ is a Riemannian manifold. Then the tension field $\tau$ of $F$ is
\begin{align}
\tau(F)&=-(2m+n)F_{*}(\mu^{kerF_{*}})+(2-n-2r)F_{*}(gradln\lambda),\label{tau}
\end{align}
where $\mu^{kerF_{*}}$ is the mean curvature vector field of the distribution of $kerF_{*}$.
\end{lemma}
From Lemma \ref{teo6} we deduce that:
\begin{theorem}
Let $F:(M^{2(m+n+r)},g_{M},J)\longrightarrow (N^{n+2r},g_{N})$ be a conformal semi-invariant submersion such that $n+2r\neq2$ where $(M,g_{M},J)$ is a K\"{a}hler manifold and $(N,g_{N})$ is a Riemannian manifold. Then any three conditions below imply the fourth:
\begin{enumerate}
\item [(i)] $F$ is harmonic
\item [(ii)] The fibres are minimal
\item [(iii)] $F$ is a horizontally homothetic map.
\end{enumerate}
\end{theorem}
We also have the following result.
\begin{corollary}
Let $F$ be a conformal semi-invariant submersion from a K\"{a}hler manifold $(M,g_{M},J)$ to a Riemannian manifold $(N,g_{N})$. If $n+2r=2$ then $F$ is harmonic if and only if the fibres are minimal.
\end{corollary}
Now we obtain necessary and sufficient condition for a conformal semi-invariant submersion to be totally geodesic. We recall that a differentiable map $F$ between two Riemannian manifolds is called totally geodesic if $$(\nabla F_{*})(X,Y)=0 \mbox{}\ \ \forall X,Y\in\Gamma(TM).$$
A geometric interpretation of a totally geodesic map is that it maps every geodesic in the total manifold into a geodesic in the base manifold in proportion to arc lengths.
We now present the following definition.
\begin{definition}
Let $F$ be a conformal semi-invariant submersion from a K\"{a}hler manifold $(M,g_{M},J)$ to a Riemannian manifold $(N,g_{N})$. Then $F$ is called a $(JD_{2},\mu)$-totally geodesic map if
$$(\nabla F_{*})(JU,X)=0, for \ U\in\Gamma(D_{2})\ and \ X\in\Gamma((kerF_{*})^\perp).$$
\end{definition}
In the sequel we show that this notion has an important effect on the geometry of the conformal submersion.
\begin{theorem}
Let $F$ be a conformal semi-invariant submersion from a K\"{a}hler manifold $(M,g_{M},J)$ to a Riemannian manifold $(N,g_{N})$. Then $F$ is  a $(JD_{2},\mu)$-totally geodesic map if and only if $F$ is horizontally homothetic map.
\end{theorem}
\begin{proof}
For $U\in\Gamma(D_{2})$ and $X\in\Gamma(\mu)$, from Lemma \ref{lem1}, we have
\begin{equation*}
(\nabla F_{*})(JU,X)=JU(ln\lambda)F_{*}X+X(ln\lambda)F_{*}JU-g_{M}(JU,X)F_{*}(gradln\lambda).
\end{equation*}
From above equation, if $F$ is a horizontally homothetic then $(\nabla F_{*})(JU,X)=0.$ Conversely, if $(\nabla F_{*})(JU,X)=0,$ we obtain
\begin{equation}
JU(ln\lambda)F_{*}X+X(ln\lambda)F_{*}JU=0.\label{e.q:4.4}
\end{equation}
Taking inner product in (\ref{e.q:4.4}) with $F_{*}JU$ and since $F$ is a conformal  submersion, we write
\begin{equation*}
g_{M}(gradln\lambda,JU)g_{N}(F_{*}X,F_{*}JU)+g_{M}(gradln\lambda,X)g_{N}(F_{*}JU,F_{*}JU)=0.
\end{equation*}
Above equation implies that $\lambda$ is a constant on $\Gamma(\mu).$ On the other hand, taking inner product in (\ref{e.q:4.4}) with $F_{*}X$, we have
\begin{equation*}
g_{M}(gradln\lambda,JU)g_{N}(F_{*}X,F_{*}X)+g_{M}(gradln\lambda,X)g_{N}(F_{*}JU,F_{*}X)=0.
\end{equation*}
From above equation, it follows that $\lambda$ is a constant on $\Gamma(JD_{2}).$ Thus $\lambda$ is a constant on $\Gamma((kerF_{*})^\perp).$ Hence proof is complete.
\end{proof}
Finally we give necessary and sufficient conditions for a conformal semi-invariant submersion to be totally geodesic.
\begin{theorem}
Let $F:(M,g_{M},J)\longrightarrow (N,g_{N})$ be a conformal semi-invariant submersion, where $(M,g_{M},J)$ is a K\"{a}hler manifold and $(N,g_{N})$ is a Riemannian manifold. $F$ totally geodesic map if and only if
\begin{enumerate}
\item [(a)] $\mathcal{C}T_{U}JV+\omega\hat{\nabla}_{U}JV=0 \ \ U,V\in\Gamma(D_{1})$
\item [(b)] $\mathcal{C}\mathcal{H}\nabla^{^M}_{U}JW+\omega T_{U}JW=0\ \ U\in\Gamma(ker F_*),\ W\in\Gamma(D_{2})$
\item [(c)] $F$ is a horizontally homothetic map.
\end{enumerate}
\end{theorem}
\begin{proof}
(a) For $U,V\in\Gamma(D_{1})$, using (\ref{e.q:2.2}), (\ref{nfixy}) and (\ref{nvw}) we have
\begin{equation*}
(\nabla F_{*})(U,V)=F_{*}(J(T_{U}JV+\hat{\nabla}_{U}JV)).
\end{equation*}
Using (\ref{e.q:3.3}) and (\ref{e.q:3.4}) in above equation we obtain
\begin{equation*}
(\nabla F_{*})(U,V)=F_{*}(\mathcal{B}T_{U}JV+\mathcal{C}T_{U}JV+\phi\hat{\nabla}_{U}JV+\omega\hat{\nabla}_{U}JV).
\end{equation*}
Since $\mathcal{B}T_{U}JV+\phi\hat{\nabla}_{U}JV\in\Gamma(kerF_{*})$, we derive
\begin{equation*}
(\nabla F_{*})(U,V)=F_{*}(\mathcal{C}T_{U}JV+\omega\hat{\nabla}_{U}JV).
\end{equation*}
Then, since $F$ is a linear isometry between $(kerF_{*})^\perp$ and $TN$, $(\nabla F_{*})(U,V)=0$ if and only if $\mathcal{C}T_{U}JV+\omega\hat{\nabla}_{U}JV=0.$\\
(b) For $U\in\Gamma(ker F_*),\ W\in\Gamma(D_{2})$, using (\ref{e.q:2.2}) and (\ref{nfixy}) we have
\begin{align*}
(\nabla F_{*})(U,W)&=\nabla^{F}_{U}F_{*}W-F_{*}(\nabla^{^M}_{U}W)\\
&=F_{*}(J\nabla^{^M}_{U}JW).
\end{align*}
Then from (\ref{nvx}) we arrive at
\begin{equation*}
(\nabla F_{*})(U,W)=F_{*}(J(T_{U}JW+\mathcal{H}\nabla^{^M}_{U}JW)).
\end{equation*}
Using (\ref{e.q:3.3}) and (\ref{e.q:3.4}) in above equation we obtain
\begin{equation*}
(\nabla F_{*})(U,W)=F_{*}(\phi T_{U}JW+\omega T_{U}JW+\mathcal{B}\mathcal{H}\nabla^{^M}_{U}JW+\mathcal{C}\mathcal{H}\nabla^{^M}_{U}JW).
\end{equation*}
Since $\phi T_{U}JW+\mathcal{B}\mathcal{H}\nabla^{^M}_{U}JW\in\Gamma(kerF_{*})$, we derive
\begin{equation*}
(\nabla F_{*})(U,W)=F_{*}(\omega T_{U}JW+\mathcal{C}\mathcal{H}\nabla^{^M}_{U}JW).
\end{equation*}
Then, since $F$ is a linear isometry between $(kerF_{*})^\perp$ and $TN$, $(\nabla F_{*})(U,W)=0$ if and only if
$\omega T_{U}JW+\mathcal{C}\mathcal{H}\nabla^{^M}_{U}JW=0.$\\
(c) For $X,Y\in\Gamma(\mu)$, from Lemma \ref{lem1}, we have
\begin{equation*}
(\nabla F_{*})(X,Y)=X(ln\lambda)F_{*}Y+Y(ln\lambda)F_{*}X-g_{M}(X,Y)F_{*}(gradln\lambda).
\end{equation*}
From above equation, taking $Y=JX$ for $X\in\Gamma(\mu)$ we obtain
\begin{align*}
(\nabla F_{*})(X,JX)&=X(ln\lambda)F_{*}JX+JX(ln\lambda)F_{*}X-g_{M}(X,JX)F_{*}(gradln\lambda)\\
&=X(ln\lambda)F_{*}JX+JX(ln\lambda)F_{*}X.
\end{align*}
If $(\nabla F_{*})(X,JX)=0,$ we obtain
\begin{equation}
X(ln\lambda)F_{*}JX+JX(ln\lambda)F_{*}X=0. \label{e.q:4.5}
\end{equation}
Taking inner product in (\ref{e.q:4.5}) with $F_{*}X$ and since $F$ is a conformal submersion, we write
\begin{equation*}
g_{M}(gradln\lambda,X)g_{N}(F_{*}JX,F_{*}X)+g_{M}(gradln\lambda,JX)g_{N}(F_{*}X,F_{*}X)=0.
\end{equation*}
Above equation implies that $\lambda$ is a constant on $\Gamma(J\mu).$ On the other hand, taking inner product in (\ref{e.q:4.5}) with $F_{*}JX$ we have
\begin{equation*}
g_{M}(gradln\lambda,X)g_{N}(F_{*}JX,F_{*}JX)+g_{M}(gradln\lambda,X)g_{N}(F_{*}X,F_{*}JX)=0.
\end{equation*}
From above equation, it follows that $\lambda$ is a constant $\Gamma(\mu).$ In a similar way, for $U,V\in\Gamma(D_{2})$, using Lemma \ref{lem1} we have
\begin{equation*}
(\nabla F_{*})(JU,JV)=JU(ln\lambda)F_{*}JV+JV(ln\lambda)F_{*}JU-g_{M}(JU,JV)F_{*}(gradln\lambda).
\end{equation*}
From above equation, taking $V=U$ we obtain
\begin{align}
(\nabla F_{*})(JU,JU)&=JU(ln\lambda)F_{*}JU+JU(ln\lambda)F_{*}JU-g_{M}(JU,JU)F_{*}(gradln\lambda)\label{e.q:4.6}\\
&=2JU(ln\lambda)F_{*}JU-g_{M}(JU,JU)F_{*}(gradln\lambda).\notag
\end{align}
Taking inner product in (\ref{e.q:4.6}) with $F_{*}JU$ and since $F$ is a conformal submersion, we derive
\begin{equation*}
2g_{M}(gradln\lambda,JU)g_{N}(F_{*}JU,F_{*}JU)-g_{M}(JU,JU)g_{N}(F_{*}(gradln\lambda),F_{*}JU)=0.
\end{equation*}
From above equation, it follows that $\lambda$ is a constant on $\Gamma(JD_{2})$. Thus $\lambda$ is a constant on $\Gamma((kerF_{*})^\perp)$. If $F$ is a horizontally homothetic map, then $F_*(grad ln \lambda)$ vanishes, thus the converse is clear, i.e. $(\nabla F_*)(X,Y)=0$ for $X, Y \in \Gamma((ker F_*)^\perp)$. Hence proof is complete.
\end{proof}

\end{document}